\documentclass[12pt]{amsart}
\usepackage{amsmath, amsthm, amssymb}
\usepackage{ifthen}
\usepackage{commath}
\usepackage{mathtools}
\usepackage{graphicx}
\usepackage{siunitx}
\usepackage{eufrak}
\usepackage{tikz-cd}
\usepackage{stix}
\usepackage{cite}
\usepackage{faktor}

\newtheorem{theorem}{Theorem}
\newtheorem*{thm}{Theorem \ref{thm}}
\newtheorem{lemma}[theorem]{Lemma}

\newtheorem{proposition}[theorem]{Proposition}
\newtheorem{corollary}[theorem]{Corollary}
\numberwithin{theorem}{section}
\usepackage{enumitem}

\DeclareMathOperator{\arccsc}{arccsc}

\newcommand{\rr}{\mathbb{R}}
\makeatletter
\let\oldabs\abs
\def\abs{\@ifstar{\oldabs}{\oldabs*}}
\let\oldnorm\norm
\def\norm{\@ifstar{\oldnorm}{\oldnorm*}}
\makeatother
\begin{document}

\title{Convex Ancient Solutions to Anisotropic Curve Shortening Flow}
\author{Theodora Bourni and Benjamin Richards}
\maketitle
\bigskip

\section*{Abstract}
We construct a translating solution to anisotropic curve shortening flow and show that for a given anisotropic factor $g:S^1\to\rr_+$, and a given direction and speed, this translator is unique. We then construct an ancient compact solution to anisotropic curve shortening flow, and show that this solution, along with the appropriate translating solution, are the unique solutions to anisotropic curve shortening flow that lie in a slab of a given width and no smaller.\\

\begin{section}{Introduction}

In what follows, $M^1$ will denote a 1-dimensional manifold, generally either $\rr$ or $S^1$, and $I$ will be some interval of the real line, possibly infinite. We say that a family of curves $X(u,t):M^1\times I\to\rr^2$ is a solution to Anisotropic Curve Shortening Flow (ACSF), with respect to the factor $g$, if\\

\[\dpd{X}{t}(u,t)=-g(\mathsf{N})\kappa(u,t) \mathsf{N}(u,t),\text{ for all }(u,t)\in M^1\times I,\]
\\
where $g$ is some smooth, positive function defined on $S^1$, $\mathsf{N}(u,t)$ is a choice of normal vector, and $\kappa(u,t)$ is the curvature with respect to this normal. We will require that our curves be embedded and strictly convex, i.e., that $\kappa>0$. We will choose $\mathsf{N}$ to be pointing towards the non-convex region of the plane. Our sign convention is such that a circle has positive curvature with respect to the outward pointing unit normal.\\

ACSF is a generalization of Curve Shortening Flow (CSF),  introduced by Taylor \cite{taylor} and Angenent-Gurtin \cite{angenentgurtin} as a physical model for certain crystal interfaces. Gage \cite{gage} studied the case where $g$ is $\pi$-periodic as a way to study regular curve shortening flow on the Euclidean plane equipped with a Minkowski norm, and proved that there is a unique self-similar solution to ACSF when $g$ exhibits this particular symmetry. While it has been proved that self-similar solutions exist without this symmetry \cite{GageLi}, it remains an open question whether these solutions are unique. \\

We say that $X(u,t)$ is an \textit{ancient solution} if $I=(-\infty,a)$ for some $a\in\rr$, and we say that it is an \textit{eternal solution} if $I=\rr$. The study of ancient solutions for CSF arose from the investigation of singularity formation, as after normalization, the limiting shape of a curve approaching a singularity is that of an ancient solution. Thus, the classifcation of all ancient solutions is a useful tool in the study of the behavior of a flow. In the case of CSF, compact, convex ancient solutions were classified by Daskalopoulos-Hamilton-Sesum \cite{DHS}, and this classification was extended to all convex cuves by Bourni-Langford-Tinaglia \cite{BLT}. Classification of convex ancient solutions for curves solving a flow based on the curvature raised to certain powers was done by Bourni-et.al \cite{betal}.\\

In this paper we will adapt and extend some of the methods used in these previous works in order to construct convex ancient solutions to ACSF that lie within a slab of a given width. We then show that the solutions constructed here are the only such ancient solutions to ACSF. This makes up the content of our main result, Theorem 5.3.\\

\begin{thm}
Let $g:S^1\to\rr_+$ be a smooth and strictly positive function, $v\in S^1$, and $w\in\mathbb{R}_+$. There exists a unique, up to translation, compact ancient solution to ACSF with respect to $g$ that lies within a slab parallel to $v$ of width $w$ and in no smaller slab. There exist two, up to translation, translating solutions to ACSF with respect to $g$ that lie within a slab parallel to $v$ of width $w$ and in no smaller slab, one that travels in the $v$ direction and one that travels in the $-v$ direction.
\end{thm}

\section{Acknowledgements}
We would like to thank Mat Langford for many valuable discussions on the topic.

Both authors were supported by grant DMS-2105026 of the National Science Foundation.
\end{section}

\begin{section}{Preliminaries}
In this section we fix some notation and calculate some useful evolution equations.\\

We will let $\theta=\theta(u,t)$ be the tangent-angle, that is the angle that the tangent to $X(M\times I)$ at $X(u,t)$ makes with the $x$-axis. We will denote this tangent by $\mathsf{T}(u,t)$, and we parametrize our curves counterclockwise. Thus we have that\\

\[\mathsf{T}=(\cos \theta,\sin \theta)\]
\\
and\\

\[\mathsf{N}=(\sin \theta,-\cos \theta).\]

We will usually use $\theta$ as the argument for $g$, and we will often abuse notation and fail to write the arguments at all. We will use $u$ for an arbitrary parametrization, and reserve the use of $s$ for arc-length parametrization.\\

We have the Frenet-Serret formulas
\begin{align*}
\dpd{\mathsf{T}}{s}=&-\kappa\mathsf{N}\\
\dpd{\mathsf{N}}{s}=&\kappa\mathsf{T}.
\end{align*}

In general, the arc-length parametrization $s$ will depend on time $t$. Thus, given a function $f$ defined on our family of curves, we have the following commutator formula\\

\[\md{f}{2}{s}{ }{t}{ }=\md{f}{2}{t}{ }{s}{ }+g\kappa^2\dpd{f}{s}.\]
\\
Moving forward, we will denote partial derivatives using subscripts, unless doing so would be made particularly annoying by the existence of other indices.\\

\begin{proposition}
We have the following evolution equations for a family of curves that satisfy ACSF.

\begin{align*}
\mathsf{T}_t=&-(g\kappa)_s\mathsf{N}\\
\mathsf{N}_t=&(g\kappa)_s\mathsf{T}\\
\theta_t=&(g\kappa)_s\\
\kappa_t=&(g\kappa)_{ss}+g\kappa^3.
\end{align*}

When our curves are strictly convex, we may parametrize our curves with respect to $\theta$. When we do so, we will let $t=\tau$, and take our partials with $\theta$ fixed instead of $u$ or $s$.  With this parametrization, we have
\[\kappa_{\tau}=\kappa^2((g\kappa)_{\theta\theta}+g\kappa).\]

If our curve is compact, and the area contained within the curve is denoted by $A(t)$, we have\\

\begin{equation}
A_t=-\int_{S^1}g(\theta)\dif \theta.\label{(0.1)}
\end{equation}

\end{proposition}

\begin{proof}

The first four evolution equations follow directly from the commutator formula and the Frenet-Serret formulas. The fifth follows by the chain rule, for we have\\

\[\kappa_t=\kappa_{\tau}+\kappa_{\theta}\theta_t=\kappa_{\tau}+\kappa\kappa_{\theta}(g\kappa)_s=\kappa_{\tau}+\kappa\kappa_{\theta}(g\kappa)_{\theta},\]
\\
while\\

\[(g\kappa)_{ss}=\kappa(\kappa(g\kappa)_{\theta})_{\theta}=\kappa\kappa_{\theta}(g\kappa)_{\theta}+\kappa^2(g\kappa)_{\theta\theta}.\]
\\
Substitution of these two into the expression for $\kappa_t$ above gives us our claim.\\

The final evolution equation is given by direct computation and Green's formula for the area inside of a closed curve.
\end{proof}

We also have a Harnack type inequality for curves undergoing ACSF.

\begin{proposition}
For a solution to ACSF defined on $[\alpha,T)$, for all $\tau\in (\alpha,T)$ we have
\begin{enumerate}[label=(\roman*)]
\item $\kappa((g\kappa)_{\theta\theta}+g\kappa)+\frac{1}{2(\tau-\alpha)}\geq 0$\\
\item $\kappa\sqrt{\tau-\alpha}$ is increasing with respect to $\tau$.
\end{enumerate}
\end{proposition}

\begin{proof}
The proof of (i) follows that in the textbook by Andrews et. al. \cite{EGF}. Defining a function\\

\[Q=\frac{(g\kappa)_\tau}{g\kappa}=\frac{\kappa_\tau}{\kappa}=\kappa((g\kappa)_{\theta\theta}+g\kappa),\]
\\
we have that its evolution is given by\\

\[Q_\tau=\kappa^2Q_{\theta\theta}+2\kappa\kappa_\theta Q_\theta+2Q^2.\]
\\
The result then follows by an ODE comparision principle with the solution $q(\tau)$  of $Q_\tau=2Q^2$:\\

\[q(t)=-\frac{1}{2(\tau-\alpha)}.\]
\\
For the proof of (ii), note that (i) and the evolution of $\kappa$ given in Proposition 1.1 implies that\\

\[(\log\kappa)_\tau+\log(\sqrt{\tau-\alpha})_\tau\geq 0.\]
\\
Since the logarithm is an increasing function, it then follows that $\kappa\sqrt{\tau-\alpha}$ must also be increasing with respect to $\tau$.
\end{proof}

As a useful corollary, we have that with a strictly convex ancient solution to ACSF, the curvature as a function of the tangent angle is nondecreasing with respect to time.

\begin{corollary}
If $X:M\times I\to\rr^2$ is a strictly convex ancient solution to ACSF, then $k_\tau\geq 0$ at all points of the solution.
\end{corollary}

\begin{proof}
We have that (i) from Proposition 1.2 is true for all $\alpha$ for which the solution is defined on $[\alpha,T)$. Since the solution is ancient, the solution is defined for all $\alpha\in (-\infty,T)$, and by taking $\alpha\to-\infty$ we have that\\

\[Q=\frac{\kappa_\tau}{\kappa}\geq 0.\]
\end{proof}

\end{section}

\begin{section}{Translators}
We now look at the existence of translators under this flow. Suppose we have a unit vector $v\in S^1$, and we take $\psi$ to be the angle $v$ makes with the $x$-axis, i.e., $v=(\cos \psi,\sin \psi)$. If we had a family of curves that translated in the direction $v$ under our flow, we could write\\

\[X_t=(\cos \psi,\sin\psi)=-g\kappa\mathsf{N}+\Phi\mathsf{T}\]
\\
where the tangential term is a result of reparametrizing. Then we have

\begin{align*}
-g\kappa=&\langle X_t,\mathsf{N}\rangle\\=&\langle (\cos \psi,\sin\psi),(\sin\theta,-\cos\theta)\rangle\\=&\cos\psi\sin\theta-\sin\psi\cos\theta\\=&\sin(\theta-\psi).\end{align*}

Solving for $\kappa$ gives us\\

\[\kappa(\theta)=\frac{\sin(\psi-\theta)}{g(\theta)}.\]

To find an initial curve for this translating solution, we write

\begin{align}
x(\theta)=&x\left(\psi-\frac{\pi}{2}\right)+\int_{\psi-\frac{\pi}{2}}^{\theta}\frac{\cos (u)}{\kappa}\dif u\nonumber\\
=&x\left(\psi-\frac{\pi}{2}\right)+\int_{\psi-\frac{\pi}{2}}^{\theta}\frac{\cos (u)g(u)}{\sin (\psi-u)}\dif u\label{(2.1)},
\end{align}
and

\begin{align}
y(\theta)&=y\left(\psi-\frac{\pi}{2}\right)+\int_{\psi-\frac{\pi}{2}}^{\theta}\frac{\sin (u)}{\kappa}\dif u\nonumber\\
&=y\left(\psi-\frac{\pi}{2}\right)+\int_{\psi-\frac{\pi}{2}}^{\theta}\frac{\sin (u)g(u)}{\sin (\psi-u)}\dif u\label{(2.2)}.
\end{align}

For convenience, we will write $x_0=x\left(\psi-\frac{\pi}{2}\right)$ and $y_0=y\left(\psi-\frac{\pi}{2}\right)$. We note that these curves are asymptotic to straight lines in the $v$ direction as $\theta\to\psi$ or $\theta\to\psi-\pi$. To see this, we take inner products and have
\begin{align*}
\langle (x(\theta),y(\theta)),v\rangle=&\langle(x(\theta),y(\theta)),(\cos \psi,\sin\psi)\rangle\\
=&x_0\cos\psi+y_0\sin\psi+\int_{\psi-\frac{\pi}{2}}^{\theta}\frac{\left(\cos (u)\cos(\psi)+\sin(u)\sin(\psi)\right)g(u)}{\sin(\psi-u)}\dif u\\
=&x_0\cos\psi+y_0\sin\psi+\int_{\psi-\frac{\pi}{2}}^\theta\cot (\psi-u)g(u)\dif u\\
\geq&x_0\cos\psi+y_0\sin\psi-(\min g)\log (\sin(\psi-\theta)).
\end{align*}
Then note that as $\theta\uparrow\psi$ or $\theta\downarrow(\psi-\pi)$, we have that $\log (\sin(\psi-\theta))\downarrow-\infty$.\\

We can also use (2) and (3) to show that the translator lives in a slab of a given width, which we will denote $w_g^{v}$. Taking $v^{\perp}=(\sin \psi,-\cos,\psi)$, we have that
\begin{align*}
\langle (x(\theta),y(\theta)),v^{\perp}\rangle=&\langle(x(\theta),y(\theta)),(\sin \psi,-\cos\psi)\rangle\\
=&x_0\sin\psi- y_0\cos\psi+\int_{\psi-\frac{\pi}{2}}^{\theta}\frac{\left(\cos (u)\sin(\psi)-\sin(u)\cos(\psi)\right)g(u)}{\sin(\psi-u)}\dif u\\
=&x_0\cos\psi+y_0\sin\psi+\int_{\psi-\frac{\pi}{2}}^\theta g(u)\dif u.\\
\end{align*}
Performing similar caclulations on $-v^{\perp}$ and combining, we find that
\begin{equation}w_g^{\psi}=\int_{\psi-\pi}^{\psi}g(u)\dif u.\label{2.3}\end{equation}

Note that if $v=e_2$, we have $\psi=\frac{\pi}{2}$.  (1) and (2) then simplify to\\

\[x(\theta)=x(0)+\int_{0}^{\theta} g(u)\dif u\hspace{1in}y(\theta)=y(0)+\int_{0}^{\theta}\tan (u)g(u)\dif u.\]
\\
In the case where $v=e_2$, we will denote $w_g^{e_2}$ by $w_g$. Note that for $g\equiv 1$, and $x(0)=y(0)=0$, our construction recovers the famous grim reaper solution for CSF.\\

\begin{proposition}
Given a smooth, strictly positive function $g:S^1\to\rr$, and vector $v\in S^1$, there exists a translating solution to ACSF that travels in the direction of $v$ with speed $1$. Furthermore, up to translations in $\rr^2$ and time, this translator is unique.
\end{proposition}

\begin{proof}
if we let $\Gamma_{g,v}$ be the initial curve defined by (2) and (3) with $v-(\cos \psi,\sin \psi)$, the family of curves given by\\

\[X(\cdot,t)=\Gamma_{g,v}+tv\]
\\
is a translating solution to ACSF, and any translating solution must be of this form, up to the addition of a constant multiple of $v$ and choice of initial points on our curve.
\end{proof}

In what follows we will mainly concern ourselves with translators in the $e_2$ and $-e_2$ directions, though by an orthogonal change of coordinates, the results will follow for translators in the $v$ and $-v$ directions as well.\\

Given a translator moving in the $e_2$ direction at speed $1$, we have by (4) that
\begin{equation}w_g=\int_{-\frac{\pi}{2}}^{\frac{\pi}{2}}g(u)\dif u\label{(2.4)}.\end{equation}
Given a vertical slab of arbitrary width $w$, we can find a translator moving in the $e_2$ direction of width $w$ by having it move in the $e_2$ direction at speed $\frac{w_g}{w}$. This gives us the following.

\begin{corollary}
Given the hypotheses of the previous proposition, and a slab of width $w$ parallel to a vector $v\in S^1$, there exists a unique translating solution to ACSF that lies in that slab and no smaller slab. Furthermore, this translator travels with speed $\frac{w^v_g}{w}$
\end{corollary}

\end{section}

\begin{section}{Compact Solutions}

This section is dedicated to proving the following Theorem.\\

\begin{theorem}
Given a smooth,strictly positive function $g:S^1\to\rr$, and a slab of width $w$ parallel to a vector $v\in S^1$, there exists a compact ancient solution to ACSF with respect to $g$ that lies in the slab and no smaller slab.
\end{theorem}

By the reasoning in previous sections, it suffices to prove this theorem for $v=e_2$ and $w=w_g$, with $w_g$ as in (4).\\

Given a smooth function $g:S^1\to\mathbb{R}^{>0}$, and using the construction developed in section 2, we will denote by $G^+$ the $t=0$ timeslice of the translator that moves at speed 1 in the $e_2$ direction as in Proposition 2.1, with $y(0)=0$ and with $x(0)$ chosen so that the curve is `centered' about the $y$-axis, i.e, $x(0)$ will be such that

\begin{equation}-\left(x(0)-\int_{0}^{\frac{-\pi}{2}}g(u)\dif u\right)=x(0)+\int_{0}^{\frac{\pi}{2}}g(u)\dif u.\label{3.1}\end{equation}

We will denote by $\{G^+_t\}_{t\in(-\infty,\infty)}$ the translating solution to ACSF that moves at speed $1$ in the $e_2$ direction such that $G^+_0=G^+$, and as above, we denote by $w_g$ the width of this translator.\\

The idea of constructing an ancient solution is as follows. We construct an appropriate sequence of ``old-but-not-ancient" solutions (these are flows that live in longer and longer time intervals) and show that one can extract a limit. In \cite{betal} the initial curves of the sequence were constructed by considering timeslices of the translating solutions further and further in the past and reflecting these across the $x$-axis to create a compact, convex curve. As our flow depends heavily on the direction of the normal vector, this method will not work for us, and we must adapt our procedure in various way to construct the sequence of flows. After constructing these so-called `old-but-not-ancient' solutions, we wish to take a limit of the corresponding flows. We will show that this limit does, in fact, exist, and that it is an ancient solution to ACSF. 

To construct the initial curves of the ``old-but-not-ancient" soluctions,  instead of reflecting our translator we will join together translators moving in the $e_2$ direction with those moving in the $-e_2$ direction. One issue that immediately arises is that those translators may lie in slabs of different width. In order to resolve this issue, we require that our translator in the $-e_2$ direction moves at a speed $\sigma$ with $\sigma$ chosen so that it lives in a slab of width $w_g$, and thus the curves match up appropriately. Note that in the construction in section 2 this corresponds to $v=\sigma(\cos \frac{3\pi}{2},\sin\frac{3\pi}{2})$, and so by (1)  and (2), we have the following expression for the initial curve of the translator that moves in the $-e_2$ direciton:\\

\[x(\theta)=x(\pi)-\frac{1}{\sigma}\int_{\pi}^{\theta}g(u)\dif u\hspace{0.1in}y(\theta)=y(\pi)-\frac{1}{\sigma}\int_{\pi}^{\theta}g(u)\tan(u)\dif u.\]
\\
We then have that the width of this translator is $\frac{1}{\sigma}\int_{\frac{\pi}{2}}^{\frac{3\pi}{2}}g(u)\dif u$, and we can solve for $\sigma$ to determine the appropriate speed. In particular, we pick $\sigma$ so that\\

\[\frac{1}{\sigma}\int_{\frac{\pi}{2}}^{\frac{3\pi}{2}}g(u)\dif u=\int_{-\frac{\pi}{2}}^{\frac{\pi}{2}}g(u)\dif u.\]
\\
Similar to the above, we will denote by $G^-$ the initial curve for the translator that moves at speed $\sigma$ in the $-e_2$ direction with $y(0)=0$ and $x(0)$ chosen so that the curve is centered as in (6), and we have $\{G^{-}_t\}_{t\in(-\infty,\infty)}$ as the corresponding translating solution to ACSF in the $-e_2$ dirction with speed $\sigma$ and with $G^{-}_0=G^{-}$.\\

Now, for $R>0$ we wish to construct a compact curve by taking some combination of $G^+_{-R}$ and $G^{-}_{-R}$. Note, however, that even though we have constructed the translators to have the same width, we have no guarantee that they will intersect the $x$-axis at the same points for any given $R$ (and, indeed, they almost certainly won't). So we cannot just take the intersections of our timeslice curves with the respective half-planes. Instead, we will take the union of both curves, and discard the pieces of the curves after their two intersections. We will denote the resulting compact curve by $G^R$. So $G^R$ is the boundary of the compact convex region bounded by $G_{-R}^{\pm}$.\\

We further note that while the resulting curve is not smooth, as a consequence of a theorem by Andrews \cite{andrews}, there exists a smooth solution to ACSF whose initial curve is $G^R$.  We will translate by time so that for every $R$ the flow becomes extinct at $t=0$, and will denote these flows by
\begin{equation}
\{G^R_t\}_{t\in(t_R,0)},\label{3.2}
\end{equation}
with\\

\[\lim_{t\downarrow t_R} G^R_t=G^R,\]
\\
with the limit being in the $C^{0}$-topology.

\begin{proposition}
Let $A_R(0)$ be the area bounded by the curve $G^R$. Then there exists a constant $C$, depending only on $g$ and $\sigma$, such that $ w_g(R+\sigma R)-C\leq A_R(0)\leq w_g(R+\sigma R)$.
\end{proposition}

\begin{proof}
Let $P$ be the rectangle $\{\envert{x}\leq w_g\}\times \{-R\leq y \leq \sigma R\}$. Note that we have $G_R\subset P$ and thus $A_R(0)\leq w_g(R+\sigma R)$.\\

To obtain the lower bound, we will estimate the areas between the two translators and the vertical edges of $P$. For the area between the curve $G^+_{-R}$ and the rectangle, we note that this area is smaller than the area between $G^+$ and the two lines defined by $x=\pm\frac{w_g}{2}$ in the halfplane $\{y\geq 0\}$. This is given by\\

\[\int_{-\frac{\pi}{2}}^{\frac{\pi}{2}}g(u)\int_{0}^{\theta}g(u)\tan (u)\dif u\dif \theta\]
\\
and we have
\begin{align*}
\int_{-\frac{\pi}{2}}^{\frac{\pi}{2}}g(u)\int_{0}^{\theta}g(u)\tan (u)\dif u\dif \theta&\leq\enVert{g}_{L^{\infty}\left(S^1\right)}^2\int_{\frac{-\pi}{2}}^{\frac{\pi}{2}}\int_{0}^{\theta}\tan (u)\dif u\dif\theta\\
&=\enVert{g}^2_{L^{\infty}\left(S^1\right)}\pi\log 2.
\end{align*}

A similar calculation for the area between $G^-_{-R}$ and the rectangle gives us an upper bound of $\sigma^{-2}\enVert{g}_{L^{\infty}\left(S^1\right)}^2\pi\log 2$, and this gives us our desired lower bound for $A_R(0)$ with $C=\enVert{g}_{L^{\infty}\left(S^1\right)}^2(1+\sigma^{-2})\pi\log 2.$
\end{proof}

We define the horizontal reach $h_R$  of the curve $G^R$ to be the distance between the two vertical supporting lines of the curve. That is, $h_R$ is such that $G_R$ is contained in a vertical slab of width $h_R$ but in no thinner vertical slab.

\begin{proposition}
There exists a constant $C$, dependent on $g$ and $\sigma$, such that for all $R\in (-\infty,0)$ we have\\

\[h_R\geq w_g-2C\arcsin \left(\frac{C}{R}\right).\]
\end{proposition}

\begin{proof}
What we will find are inner bounds for the intersection points of $G^{+}_{-R}$ and $G^{-}_{-R}$ with the $x$-axis. Note that even though there is no reason to believe that the intersection points are symmetric about the $y$-axis, our inner bounds will be. Thus, the bounds for $G^{+}_{-R}$ and $G^{-}_{-R}$ will necessarily be nested, (one set of bounds will lie within the other set) and the innermost set will therefore serve as a bound for all four intersection points. The supporting hyperplanes are certainly further from the $y$-axis as the corresponding intersection point of at least one of $G^{+}_{-R}$ and $G^{-}_{-R}$ (and exactly one unless the two curves meet at the $x$-axis), so our result will follow.\\

Let $\theta_R^{-}<\theta_R^+$ be the tangent angles for the points at which $G_{-R}^+$ intsersects the $x$-axis. Then note that we have\\

\[R=\int_0^{\theta_R^+}g(u)\tan (u)\dif u\leq\enVert{g}_{L^{\infty}\left(S^1\right)}\ln(\sec \left(\theta_R^+\right)\leq \enVert{g}_{L^{\infty}\left(S^1\right)}\sec\left(\theta_R^+\right).\]
\\
Hence $\sec \left(\theta_R^+\right)\geq\frac{R}{\enVert{g}}_{L^{\infty}\left(S^1\right)}$. A similar calculation gives the same bound for $\sec \left(\theta_R^{-1}\right)$. Then note that

\begin{align*}
\int_{\theta_R^-}^{\theta_R^+}g(u)\dif u&=w_g-\int_{-\frac{\pi}{2}}^{\theta_R^{-}}g(u)\dif u-\int_{\theta_R^+}^{\frac{\pi}{2}}g(u)\dif u\\
&\geq w_g-\enVert{g}_{L^{\infty}\left(S^1\right)}\left(\theta_R^- +\frac{\pi}{2}\right)-\enVert{g}_{L^{\infty}\left(S^1\right)}\left(\frac{\pi}{2}-\theta_R^+\right)\\
&=w_g-\enVert{g}_{L^{\infty}\left(S^1\right)} \arccsc\left(\sec\left(\theta_R^-\right)\right)-\enVert{g}_{L^{\infty}\left(S^1\right)}\arccsc\left(\sec\left(\theta_R^+\right)\right)\\
&\geq w_g-\enVert{g}_{L^{\infty}\left(S^1\right)}\arccsc\left(\frac{R}{\enVert{g}_{L^{\infty}\left(S^1\right)}}\right)-\enVert{g}_{L^{\infty}\left(S^1\right)}\arccsc\left(\frac{R}{\enVert{g}_{L^{\infty}\left(S^1\right)}}\right)\\
&=w_g-2\enVert{g}_{L^{\infty}\left(S^1\right)}\arcsin \left(\frac{\enVert{g}_{L^{\infty}\left(S^1\right)}}{R}\right).
\end{align*}

If we define $\psi_R^-$ and $\psi_R^+$ to be the tangent angles for the intersection points for $G_{-R}^-$, a similar process gives us \\

\[\int_{\psi_R^-}^{\psi_R^+}g(u)\geq w_g-2\frac{\enVert{g}_{L^{\infty}\left(S^1\right)}}{\sigma}\arcsin\left(\frac{\enVert{g}_{L^{\infty}\left(S^1\right)}}{\sigma R}\right).\]
\\
The claim then follows, with $C=\enVert{g}_{L^{\infty}\left(S^1\right)}$ or $C=\frac{\enVert{g}_{L^{\infty}\left(S^1\right)}}{\sigma}$, depending on the value of $\sigma$.
\end{proof}

Similar to our definition of $h_R$, we call $\ell_R$ the vertical reach of the curve $G^R$ and define it to be the distance between the horizontal supporting lines of the curve. We note that by construction we have $\ell_R(0)=(1+\sigma)R$.\\

We call $A_R(t)$, $h_R(t)$, $\ell_R(t)$ the enclosed area, horizontal reach, and vertical reach with respect to time of the flow defined in (7). We now want to find bounds on $A_R, h_R$, and $\ell_R$ as the curve evolves under ACSF.

\begin{proposition}
Let $C$ be the constant from Proposition 3.2.
\begin{enumerate}[label=(\roman*)]
\item $A_R(t)=-tw_g(1+\sigma)$.\\
\item $-R\leq t_R\leq -R+\frac{C}{w_g(\sigma+1)}$\\
\item $-t(1+\sigma)\leq \ell_R(t)\leq -t(1+\sigma)+\frac{2C}{w_g(1+\sigma)}$\\
\item $w_g-\frac{C}{-t(1+\sigma)+\frac{C}{w_g}}\leq h_R(t)\leq w_g$\\
\item $\kappa(\theta,t)\leq 2\left(\min_{\theta\in S^1}g(\theta)\right)\left[(1+\sigma)-\frac{C_{\kappa}}{t}\right]$, for all $t>\frac{3}{4}t_R$,\\
\end{enumerate}
where $C_{\kappa}=w_g+\frac{2C}{w_t(1+\sigma)}.$
\end{proposition}

\begin{proof}
We find (i) by integrating the evolution of A(t) from $t$ to $0$, and we obtain
\begin{align*}
A(t)=&-t\int_{S^1}g(u)\dif u\\
=&-t\left(\int_{-\frac{\pi}{2}}^{\frac{\pi}{2}}g(u)\dif u+\int_{\frac{\pi}{2}}^{\frac{3\pi}{2}}g(u)\dif u\right)\\
=&-t(w_g+\sigma w_g).
\end{align*}

The inequalities from (ii) follow from (i) and the estimates in Proposition 3.2.\\

To prove the second inequality for (iii), let $\{\Gamma_t(\theta)\}_{t\in(-\infty,\infty)}$ be the family of curves, parametrized by tangent angle, that solves ACSF by translating at speed 1 in the $e_2$ direction such that $\Gamma_t\cap \{x\leq 0\}\neq\emptyset$ for all $t\leq 0$ and $\Gamma_t\cap\{x\leq 0\}=\emptyset$ for all $t>0$. We have for all $\epsilon>0$ that $\Gamma_{-R-\epsilon}(0)<G^R(0)$ and $\Gamma_{-R-\epsilon}\cap G^R=\emptyset$. Thus, by the avoidance principle, for all $t\in(t_R,0)$ we have $\Gamma_{-R-\epsilon+(t-t_r)}\cap G^R_t=\emptyset$. Using a similar argument with a translator traveling in the $-e_2$ direction at speed $\sigma$, taking $\epsilon$ to $0$, and using the estimate for $t_R$ in (ii) gives the inequality.\\

The second part of (iv) is clear since the initial curve lies in a slab of width $w_g$. The first part of (iv) follows from noting that we must have $A_R(t)\leq h_R(t)\ell_R(t)$ by simple geometry, then using the second inequality in (iii) and the inequality in (ii). The first inequality in (iii) follows similarly using the first part of (iv).\\

To prove (v), let $\eta(\theta,t)=\langle G_R(\theta,t),\mathsf{N}(\theta,t)\rangle$ be the support function. Note that by the convexity of the curve we have $\left(\ell^2(t)+h^2(t)\right)^{\frac{1}{2}}\geq\eta(\theta,t)$ for all $\theta\in S^1$. We also have by the evolution of the curve that $-\dpd{ }{t}\eta(\theta,t)=g(\theta)\kappa(\theta,t)$. Then, by (ii) of Proposition 1.2 we have
\begin{align*}
\eta(\theta,t)=&\int_{t}^0g(\theta)\kappa(\theta,\tau)\dif \tau\\
=&\int_{t}^0g(\theta)\kappa(\theta,\tau)\frac{(\tau-t_R)^{\frac{1}{2}}}{\tau-t_R)^{\frac{1}{2}}}\dif \tau\\
\geq&g(\theta)\kappa(\theta,t)(t-t_R)^{\frac{1}{2}}\int_{t}^0\frac{1}{(\tau-t_R)}\dif \tau\\
=&g(\theta)\kappa(\theta,t)(t-t_R)\left[2\left((-t_R)^{\frac{1}{2}}-(t-t_R)^{\frac{1}{2}}\right)\right]\\
=&g(\theta)\kappa(\theta,t)(t-t_R)^{\frac{1}{2}}\frac{-2t}{(-t_R)^{\frac{1}{2}}+(t-t_R)^{\frac{1}{2}}}\\
\geq&-g(\theta)\kappa(\theta,t)t\frac{(-t_R)^{\frac{1}{2}}}{2(-t_R)^{\frac{1}{2}}}\\
=&-\frac{1}{2}g(\theta)\kappa(\theta,t)t.\\
\end{align*}

This, and our bounds for $\ell(t)$ and $h(t)$ give us our bound for $\kappa$.
\end{proof}

With these bounds, particularly the one on $\kappa(\theta,t)$, and by the parabolic evolution $\kappa_t$, we have bounds on higher derivatives of $\kappa$ as well, which allows us to take $R\to\infty$ and, passing to a subsequence as necessary, claim that there exists a limiting ancient solution to ACSF lying in a slab of width $w_g$ and no smaller.
\end{section}

\begin{section}{Uniqueness}
In this section we prove that the only convex ancient solutions that live in a given slab are the translator constructed in section 2, and the ancient solution constructed in section 3. A key result that we will need involves the asymptotic behavior of such solutions as $t\to-\infty$. In particular, we show that that the asymptotic behavior of convex ancient solutions living in a slab of width $w_g$ and no smaller slab is that of translators of width $w_g$ in the appropriate direction. This was proved for CSF in the paper by Bourni, Langford, and Tinaglia \cite{BLT}. For the proofs of the statements used in \cite{BLT}, one can also follow the book by Andrews, et. al. \cite{EGF}.\\

We define $\Pi$ to be the slab $\Pi=\{(x,y):\envert{x}<\frac{w_g}{2}\}$, and we let $\{\Gamma_t\}_{t\in(-\infty,0)}$ be a convex ancient solution, parametrized by tangent angle, that lies in the slab $\Pi$ and no smaller slab. Note that in the compact case, we have that the turning angle $\theta$ takes on values in all of $S^1$, while in the noncompact case we either have $\theta\in(-\frac{\pi}{2},\frac{\pi}{2})$ or we have $\theta\in(\frac{\pi}{2},\frac{3\pi}{2}).$\\

Define $p_{-}(t)=\Gamma_t(0)$ when $0$ is in the domain of the turning angle, and similarly define $p_{+}(t)=\Gamma_t(\pi)$. By translating in space and time, we can arrange it so that we have $y(p_{\pm}(0))=0$ in the corresponding noncompact cases, and $\lim_{t\to 0^{-}}y(p_{\pm}(t))=0$ in the compact case, where $y(p_{pm}(t)):=\langle p_{\pm}(0),e_2\rangle$.\\

As in section 3 above, we let $\{G^+_t\}_{t\in(-\infty,\infty)}$ be the translator moving at speed 1 in the positive $e_2$ direction, and we let $\{G_t ^-\}_{t\in(-\infty,\infty)}$ be the translator moving at speed $\sigma$ in the negative $e_2$ direction, so that they both lie in the slab $\Pi$ and no smaller slab.\\

We first have as preliminary results a slight adaptation of a result from \cite{BLT} (c.f., \cite{EGF}), the proofs of which follow exactly as in the $g=1$ case (CSF) considered in the paper, so we omit them here.

\begin{lemma}

The translated family $\{\Gamma^+_{s,t}\}_{t\in(-\infty,-s)}$ defined by\\

\[\Gamma_{s,t}^+=\Gamma_{t+s}-p_-(s)\]
\\
converges locally uniformly in the smooth topology as $s\to-\infty$ to the translator\\

\[\{r_-G_{r_-^{-2}t}\}_{t\in(-\infty,\infty)}\]
\\
where $r_-=\lim_{s\to-\infty}(g(0)\kappa(0,s))^{-1}.$

Similarly, the translated family $\{\Gamma^-_{s,t}\}_{t\in(-\infty,-s)}$ defined by\\

\[\Gamma_{s,t}^-=\Gamma_{t+s}-p_+(s)\]
\\
converges locally uniformly in the smooth topology as $s\to\-\infty$ to the translator\\

\[\{r_+G_{r_+^{-2}t}\}_{t\in(-\infty,\infty)}\]
\\
where $r_+=\sigma \lim_{s\to-\infty}(g(\pi)\kappa(\pi,s))^{-1}$.\\

Additionally, the solution $\{\Gamma_{t}\}_{t\in(-\infty,a)}$, where $a=0$ in the compact case and $a=\infty$ in the noncompact case, sweeps out all of $\Pi$.
\end{lemma}

Since our convex ancient solutions lie in the slab $\Pi$, it is clear that we have $r_{\pm}\leq 1$. We now show that we have $r_{\pm}=1$, or in other words, that the asymptotic translators are of maximal width. The idea, again following and adapting the proof in \cite{BLT}, is to estimate the area enclosed by the ancient solution by inscribed trapezia, then show that the rate of growth of the enclosed area of the solution as we take $t\to-\infty$  would be too great if $r<1$.

\begin{lemma}
The asymptotic translators are of maximal width, i.e., $r_{\pm}=1$.
\end{lemma}

\begin{proof}
For notational convenience, let $\kappa_{\infty}^-=\lim_{s\to-\infty}\kappa(0,s)$ and let $\kappa_{\infty}^+=\lim_{s\to-\infty}\kappa(\pi,s)$. We have by the Harnack type inquality and Proposition 3.4 that\\

\[-y(p_-(t))\geq-\left(g(0)\kappa_{\infty}^-\right)t=-r_-^{-1}t\]
\\
and\\

\[y(p_+(t))\geq- \left(g(\pi)\kappa_{\infty}^+\right)t=-\sigma r_+^{-1}t.\]
\\
First, suppose that $\{\Gamma_t\}_{t\in(-\infty,\infty)}$ is noncompact with turning angle $\theta\in(-\frac{\pi}{2},\frac{\pi}{2})$. Let $A(t)$ and $B(t)$ be two points on $\Gamma_t$ such that $y(A(t)=y(B(t)=0$ and $x(A(t)>x(B(t))$. Let $A_-(t)$ be the area enclosed by $\Gamma_t$ and the $x$-axis. We then have\\

\[-A_-'(t)=\int_{\theta(B(t))}^{\theta(A(t))}g(\xi)\dif \xi\leq w_g.\]
\\
Integrating from $-t$ to $0$, we have $A_-(t)\leq-w_gt$. Let $\delta\in(0,1)$. Since the ancient solution sweeps out the whole slab $\Pi$, we can find some $t_{\delta}<0$ such that $w_g\geq x(A(t))-x(B(t))\geq w_g-\delta$ for all $t<t_{\delta}$. Taking $t_\delta$ smaller if necessary, we can also find points $\underline{q}^{\pm}(t)\in\Gamma_t$ and a constant $C_\delta$ such that $y(\underline{q}^+(t))=y(\underline{q}^-(t))$, while $w_gr_--\delta<x(\underline{q}^+(t))-x(\underline{q}^-(t))<w_gr_-$ and $0<y(\underline{q}^{\pm}(t))-y(p_-(t))<C_\delta$.\\

The area enclosed by $\Gamma_t$ and the $x$-axis is bounded below by the area of the inscribed trapezoid with vertices $A(t), B(t)$ and $\underline{q}^{\pm}(t)$. So we have\\

\[-w_gt\geq A_-(t)\geq \frac{1}{2}(w_gr_-+w_g-2\delta)(-r_-^{-1}t-C_{\delta}).\]
\\
Multiplying both sides by $2r_-$ and rearranging, we get\\

\[C_{\delta}r_-(w_g(r_-+1)-2\delta)\geq-t(w_g(1-r_-)-2\delta)\]
\\
for all $t<t_\delta$. Taking $t\to -\infty$, we have that\\

\[w_g(1-r_-)-2\delta\leq 0\]
\\
for $\delta\in (0,1)$. Taking $\delta\to 0$ gives us that $r_-=1$.

Taking the other noncompact case, we have $-A_+'(t)\leq \sigma w_g$ and so $A_+(t)\leq -\sigma w_gt$. Taking $\delta\in (0,1)$ and defining $\overline{q}^{\pm}$ and a new constant $C_\delta$ similarly as above we get\\

\[-\sigma w_g t\geq A_-(t)\geq \frac{1}{2}(w_gr_++w_g-2\delta)(-\sigma r_+^{-1}t-C_{\delta}).\]
\\
Factoring out a $\sigma$, we proceed exactly as in the other noncompact case to get $r_+=1$. The compact case is proved by bounding the area by the sum of the two trapezia to get the inequality\\

\[-(1+\sigma)w_gt\geq-\frac{1}{2}(w_gr_{-}+w_g-2\delta)(r_{-}^{-1}t+C_{\delta})-\frac{1}{2}(w_gr_{+}+w_g-2\delta)(\sigma r_{+}^{-1}t+C_{\delta})\]
\\
for all $t<t_{\delta}$. Carrying out similar calculations as above, we have that\\

\[\left[-w_g(r_{-}^{-1}+\sigma r_{+}^{-1}-(1+\sigma))-2\delta(r_{-}^{-1}+\sigma r_{+}^{-1})\right]t\leq\left[w_g(r_{-}+r_{+}+2)-4\delta\right]C_{\delta}\]
\\
for all $t<t_{\delta}$. Taking $t\to-\infty$ implies that\\

\[w_g(r_{-}^{-1}+\sigma r_{+}^{-1}-(1+\sigma))\leq 2\delta (r_{-}^{-1}+\sigma r_{+}^{-1})\]
\\
for all $\delta>0$. Taking $\delta\to0$ then gives us the result for the compact case.
\end{proof}

\begin{theorem}
\label{thm}
Let $g:S^1\to\rr_+$ be a smooth and strictly positive function, $v\in S^1$, and $w\in\mathbb{R}_+$. There exists a unique, up to translation, compact ancient solution to ACSF with respect to $g$ that lies within a slab parallel to $v$ of width $w$ and in no smaller slab. There exist two, up to translation, translating solutions to ACSF with respect to $g$ that lie within a slab parallel to $v$ of width $w$ and in no smaller slab, one that travels in the $v$ direction and one that travels in the $-v$ direction.
\end{theorem}

\begin{proof}
Once again, we prove that this holds for $w=w_g$ and $v=e_2$. The proof is an adaptation of that found in Bourni-et. al. \cite{betal}.\\

Let $\{G_t\}_{t\in(-\infty,0)}$ be the solution constructed in section 3, and let $\{\Gamma_t\}_{t\in(-\infty,0)}$ be any other compact, convex ancient solution to ACSF that lies in the vertical slab of width $w_g$ and no smaller. Parametrize both by their respective tangent angles and define the quantities\\

\[L(t)=-\langle \Gamma_t(0),e_2\rangle+\langle\Gamma_t(\pi),e_2\rangle\]
\\
and\\
\\
\[L_0(t)=-\langle G_t(0),e_2\rangle+\langle G_t(\pi),e_2\rangle.\]
\\
Note that $L_0(t)$ corresponds to $\ell(t)$ in section 3. Since the backwards limit of our ancient solutions are both translators, curvature is nondecreasing for ancient solutions to ACSF, and $g(0)\kappa(0,t)\geq 1$ and $g(\pi)\kappa(\pi,t)\geq \sigma$ for each of our solutions, we have that\\
\\
\[\dod{ }{t}(L(t)+t(1+\sigma))\leq 0,\]
\\
so when we take a limit $t\to-\infty$, the limit exists (though it may be infinite). Let $L=\lim_{t\to\infty}L(t)$, and let $L_0=\lim_{t\to\infty}L_0(t)$, and note that by Proposition 3.4 we have that $L_0<\infty$. Note also that $-\langle \Gamma_t(0),e_2\rangle+t$ and $\langle \Gamma_t(\pi),e_2\rangle+\sigma t$ both have (possibly infinite limits) as $t\to\infty$, while $-\langle G_t(0),e_2\rangle+t$ and $\langle G_t(\pi),e_2\rangle+\sigma t$ have finite limits as $t\to\infty$.\\

We first show that $L=L_0$. To that end, suppose instead that $L>L_0$.  Let $\tilde{\Gamma_t^{\epsilon}}=\Gamma_t\big{|}_{\theta\in(-\pi,0)}-\epsilon e_1$ and let $\tilde{G_t}=G_t\big{|}_{\theta\in(-\pi,0)}$. Since $L>L_0$, there exists some $t_0$ such that $L(t)>L_0(t)$ for all $t<t_0$. Thanks to the existence of the limits above, and the fact that these limits are finite in the case of our constructed solution $\{G_t\}_{t\in(-\infty,0)}$, we can find some constant $c$ such that $y(\Gamma_t(\pi))+c>y(G_t(\pi))$ and $y(\Gamma_t(0))+c<y(G_t(0))$ for all $<t_0$. Further, since both solutions must converge at their tips to the appropriate translating solutions, there exists some $t_\epsilon<t_0$ such that $\tilde{\Gamma_{t_\epsilon}^\epsilon}+ce_2\cap\tilde{G_{t_{\epsilon}}}=\emptyset$. By the maximum principle, we then have $\tilde{\Gamma_t^{\epsilon}}+ce_2\cap\tilde{G_t}=\emptyset$ for all $t\in(t_{\epsilon},t_0)$. Taking $\epsilon\to 0$ gives us that$\tilde{\Gamma_{t_0}}+ce_2$ lies outside $\tilde{G_{t_0}}$.\\

We then repeat this argument, restricting to $\theta\in(0,\pi)$ to get that $\Gamma_{t_0}+ce_2$ lies outside $G_{t_0}$ on that side as well, and so, $G_{t_0}$ lies within $\Gamma_{t_0}$. By the strong maximum principle they cannot intersect at all, but this contradicts the fact that they both expire at time $t=0$. Thus we must have $L=L_0$.\\

Now, for $\tau>0$ define $\Gamma_t^\tau$ to be the translation of $\Gamma$ in time by $\tau$, i.e., $\Gamma_t^{\tau}=\Gamma_{t-\tau}$. Then note that $L_\tau>L=L_0$, and we can repeat the above arguement to show that $\Gamma_t^\tau$ lies outsidie of $G_t$ for all $t$. Taking $\tau\to 0$ gives that $\Gamma_t$ lies outside of $G_t$ for all $t$, but as they both expire at $t=0$, this can only happen if $\Gamma_t$ and $G_t$ coincide.\\

The proof that the translator is the unique noncompact solution that lies in a slab parallel to $v$ of width $w$ follows similarly, with the additional observation that
\[\lim_{\theta\to\frac{\pi}{2}^+}\tilde{\Gamma}_t^{\epsilon}(\theta)<-\frac{w_g}{2}\]
for all $\epsilon>0$, which allows us to use the avoidance principle despite both curves involved being noncompact.
\end{proof}

\end{section}

\bibliographystyle{acm}
\bibliography{ACSFbib}
\end{document}